\documentclass{article} 
\usepackage{amsthm,amsmath,amsfonts,amssymb, hyperref,cases,xcolor,authblk}
\usepackage[capitalise]{cleveref}
\usepackage{vmargin}
\setmarginsrb{2cm}{1cm}{2cm}{3cm}{1cm}{1cm}{2cm}{2cm}

\numberwithin{equation}{section}
\newtheorem{Theorem}{Theorem}[section]

\newtheorem{Proposition}[Theorem]{Proposition}

\newtheorem{Remark}[Theorem]{Remark}

\renewcommand{\leq}{\leqslant}

\renewcommand{\geq}{\geqslant}

\renewcommand{\div}{\operatorname{div}}

\title{Asymptotic analysis of an optimal control problem for a viscous incompressible fluid with Navier slip boundary conditions}    
\author[1]{Claudia Gariboldi}
\author[2]{Tak\'eo Takahashi}
\affil[1]{Departamento de Matem\'atica, FCEFQyN,
Universidad Nacional de R\'io Cuarto, 

Ruta 36 Km 601, 5800 R\'io Cuarto, Argentina
}
\affil[2]{Universit\'e de Lorraine, CNRS, Inria, IECL, F-54000 Nancy, France}

\date{\today}

\begin{document}                  

\maketitle    

\abstract{
We consider an optimal control problem for the Navier-Stokes system with Navier slip boundary conditions. 
We denote by $\alpha$ the friction coefficient and
we analyze the asymptotic behavior of such a problem as $\alpha\to \infty$.
More precisely, we prove that if we take an optimal control for each $\alpha$, 
then there exists a sequence of optimal controls converging to an optimal control of the same optimal
control problem for the Navier-Stokes system with the Dirichlet boundary condition.
We also show the convergence of the corresponding direct and adjoint states.
}

\vspace{1cm}



\section{Introduction}
In this article, we study an optimal control problem associated with the Navier-Stokes system. This classical system is a standard model for the motion of a viscous incompressible fluid. It is also usual to assume that the fluid adheres to the exterior boundary and to  consider thus the no-slip boundary condition. 
Nevertheless, in some physical situations, one can also consider the Navier slip boundary condition introduced by Navier in \cite{navier1823}, see for instance  \cite{Jager}, \cite{Kistler}, \cite{MR2123407}, \cite{Malek}, etc. 
Recently, several studies have been done in the case of fluid-structure interaction systems and in particular in \cite{MR3281946}, the authors show that with this boundary condition, one can recover the collisions between rigid bodies that are absent with the Dirichlet boundary condition (see \cite{Hillairet}, \cite{HT}).
Finally, a rigorous derivation of this condition from the Boltzmann equation is done in \cite{CoronF}.
The Navier slip boundary condition allows the fluid to slip tangentially to the boundary and involves a friction coefficient associated with this motion. Formally, if this coefficient goes to infinity, one recover the classical no-slip boundary condition.

Our aim here is to compare the optimal control problems for these two boundary conditions and to prove an asymptotic property as the friction coefficient goes to infinity. 
In order to show such a convergence, one has to first consider an appropriate functional framework. One possible choice is to consider weak solutions since in that case one can prove existence of global solutions without restrictions on the size of data and of the controls. However, in that case, the uniqueness of solutions is an open problem and the optimal control problems can not be stated properly if we do not know which state has to be used for the criterion.

Consequently, we work here with strong solutions and we thus restrict the size of the data to get existence and uniqueness of our solutions in $(0,T)$ where $T>0$ is given.
In order to start our asymptotic analysis, we however first need to show that this restriction is uniform with respect to the friction coefficient in the Navier boundary condition.

Another difficulty comes from the fact that since the system of Navier-Stokes is nonlinear, the optimal controls are not unique. Our main result thus states that, given a family of optimal controls, one can extract a subsequence converging towards an optimal control of the Navier-Stokes system with Dirichlet boundary conditions.

Let us give our precise notation and hypotheses:
first we consider $\Omega\subset \mathbb{R}^3$ a bounded domain of class $C^{2,1}$ and we write
the Navier-Stokes system with Dirichlet boundary conditions:
\begin{equation}\label{ns0.0}
\left\{
\begin{array}{rl}
\partial_t u +(u\cdot \nabla) u- \div \sigma(u,p)= f 1_{\omega} & \text{in} \ (0,T)\times \Omega,\\
\div u = 0 & \text{in} \ (0,T)\times \Omega,\\
u=b & \text{on} \ (0,T)\times \partial \Omega,\\
u(0,\cdot)=a & \text{in} \ \Omega.
\end{array}
\right.
\end{equation}
In the above system, $u$ and $p$ are respectively the fluid velocity and the pressure of the fluid. The functions $a$ and $b$ are respectively the initial and the boundary conditions that are fixed in this work. The control $f$ is acting in the open non empty set $\omega\Subset \Omega$, and 
the optimal control problem we consider is
\begin{equation}\label{ns0.1}
J(\overline{f})=\inf_{f\in \mathcal{U}} J(f),
\end{equation}
where
\begin{equation}\label{ns0.2}
J(f):=\frac 12 \int_0^T\int_{\Omega} |u_f-z_d|^2 \ dx\ dt+\frac{M}{2} \int_0^T\int_\omega |f|^2 \ dx\ dt.
\end{equation}
Here
$$
M>0, \quad z_d\in L^2(0,T;L^2(\Omega))
$$
and $\mathcal{U}$ is a subset of $L^2(0,T;L^2(\omega))$. The choice of $\mathcal{U}$ and of the data $(a,b)$ has to be done in such a way that the system \eqref{ns0.0} admits a unique solution $(u_f,p_f)$ so that the functional $J$ is well-defined.

We assume in particular
\begin{equation}\label{cc1}
a\in H^1(\Omega), \quad \div a=0, \quad a=b(0,\cdot) \quad \text{on} \ \partial \Omega
\end{equation}
and
\begin{equation}\label{cc2}
b\in L^2(0,T;H^{3/2}(\partial \Omega))\cap H^{1/4}(0,T;L^2(\partial \Omega)), \quad
\int_{\partial \Omega} b(t,\cdot)\cdot \nu \ d\gamma=0 \quad (t\in [0,T]).
\end{equation}

In \eqref{ns0.0}, we have denoted by $\sigma(u,p)$ the Cauchy stress tensor:
$$
\sigma(u,p):=2\mu D(u)-pI_3, \quad D(u):=\frac 12 \left(\nabla u +(\nabla u)^\top\right).
$$

Let us now consider the corresponding optimal control problem when we replace the Dirichlet boundary condition in \eqref{ns0.0}. 
by the Navier slip boundary condition. In that case our system writes
\begin{equation}\label{ns1.0}
\left\{
\begin{array}{rl}
\partial_t u +(u\cdot \nabla) u- \div \sigma(u,p)= f 1_{\omega} & \text{in} \ (0,T)\times \Omega,\\
\div u = 0 & \text{in} \ (0,T)\times \Omega,\\
u\cdot \nu=b\cdot \nu & \text{on} \ (0,T)\times \partial \Omega,\\
\left[2\mu D(u)\nu+\alpha (u-b)\right]_\tau = 0 & \text{on} \ (0,T)\times \partial \Omega,\\
u(0,\cdot)=a & \text{in} \ \Omega.
\end{array}
\right.
\end{equation}
We have denoted by $\nu$ the unit normal vector exterior to $\partial \Omega$ and by $w_\tau$ the tangential component of a vector $w\in \mathbb{R}^3$:
$$
w_\tau:=w-(w\cdot \nu)\nu = \nu\times (w\times \nu).
$$
In the second boundary condition of the above system, the parameter $\alpha\geq 0$ is the coefficient of friction of the Navier boundary condition.

For the above system, we can also consider the optimal control problem
\begin{equation}\label{nh1.1}
J_\alpha(\overline{f_\alpha})=\inf_{f\in \mathcal{U}} J_\alpha(f) \quad (\alpha\geq 0),
\end{equation}
where
\begin{equation}\label{nh1.2}
J_\alpha(f)=\frac 12 \int_0^T\int_{\Omega} |u_{\alpha,f}-z_d|^2 \ dx\ dt+\frac{M}{2} \int_0^T\int_\omega |f|^2 \ dx\ dt \quad (\alpha\geq 0).
\end{equation}
As in the first case, one have to choose the data and the set $\mathcal{U}$ in such a way that the systems \eqref{ns1.0} admits a unique solution
$(u_{\alpha,f},p_{\alpha,f})$ for any $f\in \mathcal{U}$ and for any $\alpha$ large enough.

In \cref{sec_pre}, we show that we can take $\mathcal{U}$ as a small ball of $L^2(0,T;L^2(\omega))$ independently of $\alpha$.
This allows us to then study the asymptotic behavior of the optimal control problems as the coefficient of friction $\alpha$ goes to infinity.

In what follows, we write
\begin{equation}\label{1448}
V_0^1:=\left\{u\in H^1(\Omega) \ ; \ \div u=0, \quad u=0 \ \text{on} \ \partial \Omega\right\},
\end{equation}
$$
V_0^{-1}:=(V_0^1)'.
$$

We are now in position to state our main result:
\begin{Theorem}\label{Thmain}
Assume that $(a,b)$ and $\mathcal{U}$ satisfy the above hypotheses and are such that for any $f\in \mathcal{U}$, the systems \eqref{ns0.0}
and \eqref{ns1.0} are well-posed on $[0,T]$ with the properties \eqref{ns7.0ter}, \eqref{ns7.0bis}.

Then for any $\alpha$ large enough, the problem \eqref{nh1.1} admits a solution $\overline{f_\alpha}$ and
there exist $\overline{f}\in \mathcal{U}$ and 
a sequence such that as $\alpha \to \infty$
\begin{equation}\label{ns5.2}
\overline{f_\alpha} \to \overline{f} \quad \text{strongly in} \quad L^2(0,T;L^2(\omega)),
\end{equation}
and $\overline{f}$ is a solution of \eqref{ns0.1}.
Moreover, the corresponding solutions $(u_{\overline{f}},p_{\overline{f}})$ and $(u_{\alpha,\overline{f_\alpha}},p_{\alpha,\overline{f_\alpha}})$
of \eqref{ns0.0} and \eqref{ns1.0}
satisfy
\begin{equation}\label{ns5.3}
u_{\alpha,\overline{f_\alpha}} \rightharpoonup
 u_{\overline{f}} \quad \text{weakly * in}\quad L^2(0,T;H^1(\Omega))\cap L^\infty(0,T;L^2(\Omega)),
\end{equation}
\begin{equation}\label{ns5.4}
\partial_t u_{\alpha,\overline{f_\alpha}} \rightharpoonup
\partial_t u_{\overline{f}} \quad \text{weakly in}\quad L^{4/3}(0,T;V_0^{-1}),
\end{equation}
\begin{equation}\label{ns5.5}
u_{\alpha,\overline{f_\alpha}} \to
 u_{\overline{f}} \quad \text{strongly in}\quad L^2(0,T;L^2(\Omega)),
\end{equation}
\begin{equation}\label{ns5.6}
\sqrt{\alpha} (u_{\alpha,\overline{f_\alpha}}-b) \to
 0 \quad \text{strongly in}\quad L^2(0,T;L^2(\partial \Omega)).
\end{equation}
The solutions of the adjoint systems (defined by \eqref{ns3.2} and \eqref{ns3.3}) $(\phi_{f},\pi_{f})$ and $(\phi_{\alpha,f},\pi_{\alpha,f})$
satisfy
\begin{equation}\label{ns5.7}
\phi_{\alpha,\overline{f_\alpha}} \rightharpoonup
 \phi_{\overline{f}} \quad \text{weakly * in}\quad L^2(0,T;H^1(\Omega))\cap L^\infty(0,T;L^2(\Omega)),
\end{equation}
\begin{equation}\label{ns5.8}
\partial_t \phi_{\alpha,\overline{f_\alpha}} \rightharpoonup
\partial_t \phi_{\overline{f}} \quad \text{weakly in}\quad L^{4/3}(0,T;V_0^{-1}),
\end{equation}
\begin{equation}\label{ns5.9}
\phi_{\alpha,\overline{f_\alpha}} \to
 \phi_{\overline{f}} \quad \text{strongly in}\quad L^2(0,T;L^2(\Omega)),
\end{equation}
\begin{equation}\label{ns5.6bis}
\sqrt{\alpha} \phi_{\alpha,\overline{f_\alpha}} \to
 0 \quad \text{strongly in}\quad L^2(0,T;L^2(\partial \Omega)).
\end{equation}
\end{Theorem}
The interest of the adjoint systems with respect to problems \eqref{ns0.1} and \eqref{nh1.1} is given in \cref{T01}: they are associated with the first order condition of the optimal control problems. 

The result given in \cref{Thmain} is in the same spirit as previous results obtained for other partial differential equations: 
\cite{MR1974387}, \cite{MR2569618}, \cite{MR2477059}, \cite{MR3410698},   
in the case of elliptic problems and \cite{MR3083990}, \cite{MenaldiTarzia},  in the case of parabolic systems.

The outline of the paper is as follows: in \cref{sec_pre}, we show that the hypotheses of \cref{Thmain} can be satisfied for $(a,b)$ small enough and $\mathcal{U}$ as a small ball of $L^2(0,T;L^2(\omega))$. Then \cref{sec_asy} is devoted to results of convergence as $\alpha \to \infty$ of the solutions of \eqref{ns1.0} and of the adjoint systems. These results allow us to reduce the proof of the main result to the convergence of the family of the optimal controls. In \cref{sec_exi}, we show that 
 \eqref{ns0.1} and \eqref{nh1.1} admit at least a solution and we give the first order condition in terms of the adjoint systems  \eqref{ns3.2} and \eqref{ns3.3}.
 Finally in \cref{sec_main} we gather the previous results and prove \cref{Thmain}.
In \cref{sec2d}, we present the bidimensional case, where the hypotheses on the data are weaker.

\section{Uniform well-posedness of the Navier-Stokes systems}\label{sec_pre}
With the hypotheses of the introduction, in particular \eqref{cc1} and \eqref{cc2}, 
classical results yield the existence and uniqueness of strong solutions for the system \eqref{ns0.0}
and for the system \eqref{ns1.0} for data small enough (that is $a$, $b$ and $f$). 
Nevertheless, here we have to take care that the smallness conditions do not depend on $\alpha$ and we thus need to derive standard a priori estimates for the system \eqref{ns1.0} to show the uniformity of our conditions.
To simplify, we take $\alpha$ large enough and in particular satisfying 
\begin{equation}\label{1556}
\alpha>\| b\|_{L^\infty(0,T;L^\infty(\partial \Omega))}+1.
\end{equation}

Let us assume that for some $\widehat{f}\in L^2(0,T;L^2(\omega))$ there exists a unique smooth solution $(\widehat{u},\widehat{p})$ of \eqref{ns0.0} on $[0,T]$. We show that in a neighborhood of $\widehat{f}$, independent of $\alpha$, the systems \eqref{ns0.0}
and \eqref{ns1.0} are well-posed on $[0,T]$. 

\begin{Proposition}\label{P01}
Assume $T>0$ and that \eqref{ns0.0} admits a strong solution $(\widehat{u},\widehat{p})$ with 
$$
\widehat{u}\in H^1(0,T;H^2(\Omega)).
$$
There exists a constant $\widehat{C}$ independent of $\alpha$ such that if 
\begin{equation}\label{ns1.4}
\|\widehat{u}\|_{H^1(0,T;H^2(\Omega))}\leq \widehat{C},
\end{equation}
then there exists $\varepsilon$ such that for any $f\in L^2(0,T;L^2(\omega))$,
$$
\|f-\widehat{f}\|_{L^2(0,T;L^2(\omega))}\leq \varepsilon,
$$
the system \eqref{ns1.0} admits a unique strong solution 
\begin{equation}\label{ns7.0ter}
(u_{\alpha,f},p_{\alpha,f})\in \left[L^2(0,T;H^2(\Omega))\cap C^0([0,T];H^1(\Omega)) \cap H^1(0,T;L^2(\Omega)) \right]\times L^2(0,T;H^1(\Omega)/\mathbb{R}).
\end{equation}
Moreover there exists a constant $C$ independent of $\alpha$ such that 
\begin{equation}\label{ns7.0bis}
\|u_{\alpha,f}\|_{L^\infty(0,T;H^1(\Omega))}
+
\|\mathbb{P}\Delta u_{\alpha,f}\|_{L^2(0,T;L^2(\Omega))}
+
\|u_{\alpha,f}\|_{L^2(0,T;W^{1,6}(\Omega))}
\leq C.
\end{equation}
\end{Proposition}
\begin{proof}
By using standard results (see, for instance \cite{GrubbSolonnikov}), there exists a unique local strong solution 
$$
(u,p)=(u_{\alpha,f},p_{\alpha,f})\in \left[L^2(0,T_{loc};H^2(\Omega))\cap C^0([0,T_{loc}];H^1(\Omega)) \cap H^1(0,T_{loc};L^2(\Omega)) \right]\times L^2(0,T_{loc};H^1(\Omega)/\mathbb{R})
$$ 
of the system \eqref{ns1.0} and it exists as long as the $H^1(\Omega)$-norm of $u(t,\cdot)$ remains bounded. 
We thus only need to estimate the $H^1(\Omega)$-norm of $u(t,\cdot)$. 
We consider 
$$
w:=u-\widehat{u}, \quad q:=p-\widehat{p}, \quad g:=f-\widehat{f}
$$
that satisfy
\begin{equation}\label{ns2.0}
\left\{
\begin{array}{rl}
\partial_t w +(w\cdot \nabla) w
+(\widehat{u}\cdot \nabla) w
+(w\cdot \nabla) \widehat{u}
- \div \sigma(w,q)
= g 1_{\omega} & \text{in} \ (0,T)\times \Omega,\\
\div w = 0 & \text{in} \ (0,T)\times \Omega,\\
w\cdot \nu=0 & \text{on} \ (0,T)\times \partial \Omega,\\
\left[2\mu D(w)\nu+\alpha w\right]_\tau = -\left[2\mu D(\widehat{u})\nu\right]_\tau  & \text{on} \ (0,T)\times \partial \Omega,\\
w(0,\cdot)=0 & \text{in} \ \Omega.
\end{array}
\right.
\end{equation}
To obtain our estimates, we multiply the first equation of \eqref{ns2.0} by $w$ and we deduce
\begin{multline}\label{ns4.1}
\frac{1}{2}\frac{d}{dt} \int_{\Omega} |w|^2 \ dx + \int_{\partial \Omega} \frac{b\cdot \nu}{2} w_\tau^2 \ d\gamma
+\int_{\Omega} [(w\cdot\nabla)\widehat{u}]\cdot w \ dx + 2\mu \int_{\Omega} |D(w)|^2 \ dx +\alpha \int_{\partial \Omega} w_\tau^2 \ d\gamma
\\
=\int_{\omega} g\cdot w \ dx -  2\mu \int_{\partial \Omega} \left[D(\widehat{u})\nu\right]_\tau \cdot w_\tau \ d\gamma.
\end{multline}

Using H\"older's inequality, the Sobolev embedding $H^1(\Omega)\subset L^6(\Omega)$ and the Korn inequality, we deduce
\begin{multline}\label{1558}
\left| \int_{\Omega} [(w\cdot\nabla)\widehat{u}]\cdot w \ dx\right|
\leq \|w\|_{L^6(\Omega)}\|\nabla \widehat{u}\|_{L^3(\Omega)} \|w\|_{L^2(\Omega)}
\leq C \|w\|_{H^1(\Omega)}\|\widehat{u}\|_{H^1(\Omega)}^{1/2} \|\widehat{u}\|_{H^2(\Omega)}^{1/2} \|w\|_{L^2(\Omega)}
\\
\leq \mu \|D(w)\|_{L^2(\Omega)}^2+
C(1+\|\widehat{u}\|_{H^1(\Omega)} \|\widehat{u}\|_{H^2(\Omega)}) \|w\|_{L^2(\Omega)}^2.
\end{multline}

Using \eqref{1556},
we deduce from \eqref{ns4.1} and from \eqref{1558} that
\begin{multline}\label{ns1.2}
\int_{\Omega} |w(t,\cdot)|^2 \ dx 
+ 2\mu \int_0^t \int_{\Omega} |D(w)|^2 \ dx \ ds
+\alpha \int_0^t \int_{\partial \Omega} w_\tau^2 \ d\gamma \ ds
\\
\leq 
C\left(\|g\|_{L^2(0,T;L^2(\omega))}^2 +\|\widehat{u}\|_{L^2(0,T;H^2(\Omega))}^2\right)
+C(1+\|\widehat{u}\|_{L^\infty(0,T;H^2(\Omega))}^2)\int_0^t \int_{\Omega} |w|^2 \ dx \ ds,
\end{multline}
where the constants $C$ are independent of $\alpha$.

Using the Gr\"onwall lemma, we deduce 
\begin{multline}
\|w\|_{L^\infty(0,T;L^2(\Omega))}^2
+\|w\|_{L^2(0,T;H^1(\Omega))}^2
+\alpha \|w_\tau\|_{L^2(0,T;L^2(\partial \Omega))}^2
\\
\leq 
C\left(\|g\|_{L^2(0,T;L^2(\omega))}^2 +\|\widehat{u}\|_{L^2(0,T;H^2(\Omega))}^2\right)
\exp\left(C(1+\|\widehat{u}\|_{L^\infty(0,T;H^2(\Omega))}^2)T\right),
\end{multline}
where the constants $C$ are independent of $\alpha$.

Then, we multiply the first equation of \eqref{ns2.0} by $-\mathbb{P}\mu\Delta w$ where $\mathbb{P}$ is the Leray projector. We obtain after integration by parts
\begin{multline}\label{ns0.7}
\frac{d}{dt}\left(
\mu \int_{\Omega} |Dw|^2 \ dx
+\frac{\alpha}{2} \int_{\partial \Omega} |w_\tau|^2 \ d\gamma
+\int_{\partial \Omega} w_\tau\cdot \left[2\mu D(\widehat{u})\nu\right]_\tau \ d\gamma
\right)
-\int_{\partial \Omega} w_\tau\cdot \left[2\mu D(\partial_t\widehat{u})\nu\right]_\tau \ d\gamma
\\
+\int_{\Omega} (w\cdot \nabla) w \cdot (-\mu\mathbb{P}\Delta w) \ dx
+\int_{\Omega} (\widehat{u}\cdot \nabla) w \cdot (-\mu\mathbb{P}\Delta w) \ dx
+\int_{\Omega} (w\cdot \nabla) \widehat{u} \cdot (-\mu \mathbb{P}\Delta w) \ dx
\\
+ \int_{\Omega} |\mu \mathbb{P}\Delta w|^2 \ dx
= \int_{\omega} g \cdot (-\mu\mathbb{P}\Delta w) \ dx.
\end{multline}

Note that in $(0,T)$,
\begin{equation}\label{ns0.3}
\left\{
\begin{array}{rl}
-\mu \Delta w+\nabla Q=-\mu \mathbb{P}\Delta w & \text{in} \ \Omega,\\
\div w = 0 & \text{in} \ \Omega,\\
w\cdot \nu=0 & \text{on} \ \partial \Omega,\\
\left[2\mu D(w)\nu+\alpha w\right]_\tau = -\left[2\mu D(\widehat{u})\nu\right]_\tau  & \text{on} \ \partial \Omega,
\end{array}
\right.
\end{equation}
and thus from Theorem 2.2 in \cite{AACG} (see also \cite{AACG-Cras}),
\begin{equation}\label{ns0.4}
\|w\|_{W^{1,6}(\Omega)}\leq C \left( \|\mu \mathbb{P}\Delta w\|_{L^2(\Omega)}+\|\widehat{u}\|_{H^2(\Omega)}\right)
\end{equation}
where $C$ is independent of $\alpha$.
We deduce from \eqref{ns0.4}
\begin{multline}\label{ns0.5}
\left|\int_{\Omega} (w\cdot \nabla) w \cdot (-\mu\mathbb{P}\Delta w) \ dx\right|
\leq \|w\|_{L^6(\Omega)} \|\nabla w\|_{L^2(\Omega)}^{1/2} \|\nabla w\|_{L^6(\Omega)}^{1/2} \|\mu \mathbb{P}\Delta w\|_{L^2(\Omega)}
\\
\leq C \|w\|_{H^1(\Omega)}^{3/2} \left( \|\mu \mathbb{P}\Delta w\|_{L^2(\Omega)}^{3/2}
+\|\widehat{u}\|_{H^2(\Omega)}^{1/2} \|\mu \mathbb{P}\Delta w\|_{L^2(\Omega)}\right)
\\
\leq C\left(\|w\|_{H^1(\Omega)}^{6} + \|\widehat{u}\|_{H^2(\Omega)}^2 \right) +\frac{1}{8} \|\mu \mathbb{P}\Delta w\|_{L^2(\Omega)}^2
\end{multline}
and similarly,
\begin{multline}\label{ns0.6}
\left|
\int_{\Omega} (\widehat{u}\cdot \nabla) w \cdot (-\mu\mathbb{P}\Delta w) \ dx
+\int_{\Omega} (w\cdot \nabla) \widehat{u} \cdot (-\mu \mathbb{P}\Delta w) \ dx
\right| 
\\
\leq C\left(\|\widehat{u}\|_{H^1(\Omega)}^6 +\|w\|_{H^1(\Omega)}^{6} + \|\widehat{u}\|_{H^2(\Omega)}^2\right) 
+\frac{1}{8} \|\mu \mathbb{P}\Delta w\|_{L^2(\Omega)}^2.
\end{multline}

Combining \eqref{ns0.5} and \eqref{ns0.6} with \eqref{ns0.7}, we deduce
\begin{multline}\label{ns0.8bis}
\mu \int_{\Omega} |Dw(t,\cdot)|^2 \ dx
+\frac{\alpha}{2} \int_{\partial \Omega} |w_\tau(t,\cdot)|^2 \ d\gamma
+\int_{\partial \Omega} w_\tau(t,\cdot)\cdot \left[2\mu D(\widehat{u}(t,\cdot))\nu\right]_\tau \ d\gamma
\\
-\int_0^t \int_{\partial \Omega} w_\tau\cdot \left[2\mu D(\partial_t\widehat{u})\nu\right]_\tau \ d\gamma \ ds
+\frac{1}{2}\int_0^t \int_{\Omega} |\mu \mathbb{P}\Delta w|^2 \ dx
\\
\leq C\left(\|g\|_{L^2(0,T;L^2(\omega))}^2 +\|\widehat{u}\|_{L^\infty(0,T;H^1(\Omega))}^6 + \|\widehat{u}\|_{L^2(0,T;H^2(\Omega))}^2
+\int_0^t \|w(s,\cdot)\|_{H^1(\Omega)}^{6} \ ds \right).
\end{multline}
The above estimate combined with \eqref{ns1.2}, Korn's inequality and trace properties yields
\begin{multline}
\|w(t,\cdot)\|_{H^1(\Omega)}^{2}
+\frac{\alpha}{2} \int_{\partial \Omega} |w_\tau(t,\cdot)|^2 \ d\gamma
\\
+ 2\mu \int_0^t \int_{\Omega} |D(w)|^2 \ dx \ ds
+\frac{\alpha}{2}  \int_0^t \int_{\partial \Omega} w_\tau^2 \ d\gamma \ ds
+\frac{1}{2}\int_0^t \int_{\Omega} |\mu \mathbb{P}\Delta w|^2 \ dx\ ds
\\
\leq 
C_1\left(
\|g\|_{L^2(0,T;L^2(\omega))}^2 
+\|\widehat{u}\|_{L^\infty(0,T;H^1(\Omega))}^6 
+ \|\widehat{u}\|_{H^1(0,T;H^2(\Omega))}^2
\right)
\\
+C_2(1+\|\widehat{u}\|_{L^\infty(0,T;H^2(\Omega))}^2)\left(\int_0^t \|w(s,\cdot)\|_{L^2(\Omega)}^{2} \ ds
+\int_0^t \|w(s,\cdot)\|_{H^1(\Omega)}^{6} \ ds \right),
\end{multline}
where $C_1$, $C_2$ are independent of $\alpha$.

Using Gr\"onwall's lemma, we deduce that if
\begin{multline}\label{ns1.3}
C_1\left(
\|g\|_{L^2(0,T;L^2(\omega))}^2 
+\|\widehat{u}\|_{L^\infty(0,T;H^1(\Omega))}^6 
+ \|\widehat{u}\|_{H^1(0,T;H^2(\Omega))}^2
\right)
\exp\left(C_2(1+\|\widehat{u}\|_{L^\infty(0,T;H^2(\Omega))}^2)T\right)
\\
< \sqrt{\frac{\ln(2)}{T}}
\end{multline}
then 
\begin{multline*}
\forall t\in [0,T], \quad \|w(t,\cdot)\|_{H^1(\Omega)}^{2}
\\
\leq 2C_1\left(
\|g\|_{L^2(0,T;L^2(\omega))}^2 
+\|\widehat{u}\|_{L^\infty(0,T;H^1(\Omega))}^6 
+ \|\widehat{u}\|_{H^1(0,T;H^2(\Omega))}^2
\right)
\exp\left(C_2(1+\|\widehat{u}\|_{L^\infty(0,T;H^2(\Omega))}^2)t\right)
\end{multline*}
and thus remains bounded on $[0,T]$ by a constant independent of $\alpha$.
This concludes the proof.
\end{proof}

\begin{Remark}
With the conditions \eqref{cc1} and \eqref{cc2}, classical results (see for instance \cite{Temam}, \cite{ConstantinFoias}, \cite{GrubbSolonnikov0})
give the existence and uniqueness of 
a strong solution $(\widehat{u},\widehat{p})$ of \eqref{ns0.0} associated with $\widehat{f}$ provided that
$$
\|a\|_{H^1(\Omega)}+\|b\|_{L^2(0,T;H^{3/2}(\partial \Omega))\cap H^{1/4}(0,T;L^2(\partial \Omega))}
+\|\widehat{f}\|_{L^2(0,T;L^2(\omega))}
$$
is small enough. To obtain a stronger solution and the bound \eqref{ns1.4} one needs stronger hypotheses on $a$, $b$ and $\widehat{f}$
(see \cite{GrubbSolonnikov0}). To simplify the presentation, we only keep here the condition \eqref{ns1.4}  as the hypothesis.
\end{Remark}

In what follows
\begin{equation}\label{ns1.5}
\mathcal{U}:=\left\{f\in L^2(0,T;L^2(\omega)) \ ; \ \|f-\widehat{f}\|_{L^2(0,T;L^2(\omega))}\leq \varepsilon\right\},
\end{equation}
where $\varepsilon$ is given by \cref{P01}.

For any $f\in \mathcal{U}$, we denote by $(u_f,p_f)$ the strong solution of \eqref{ns0.0} on $(0,T)$ and by  
$(u_{\alpha,f},p_{\alpha,f})$ the strong solution of \eqref{ns1.0} on $(0,T)$.

\section{Adjoint systems and convergence as $\alpha\to \infty$}\label{sec_asy}
In this section, we define the adjoint systems for the optimal control problems \eqref{ns0.1} and \eqref{nh1.1} and we show convergences results for the direct state and the adjoint state as $\alpha\to \infty$.

First let us define the adjoint systems of \eqref{ns0.0} and of \eqref{ns1.0}:
\begin{equation}\label{ns3.2}
\left\{
\begin{array}{rl}
-\partial_t \phi +(\nabla u_f)^\top \phi -(u_f\cdot \nabla) \phi- \div \sigma(\phi,\pi)= (u_f-z_d) & \text{in} \ (0,T)\times \Omega,\\
\div \phi = 0 & \text{in} \ (0,T)\times \Omega,\\
\phi=0 & \text{on} \ (0,T)\times \partial \Omega,\\
\phi(T,\cdot)=0 & \text{in} \ \Omega,
\end{array}
\right.
\end{equation}
and
\begin{equation}\label{ns3.3}
\left\{
\begin{array}{rl}
-\partial_t \phi +(\nabla u_{\alpha,f})^\top \phi -(u_{\alpha,f}\cdot \nabla) \phi- \div \sigma(\phi,\pi)= (u_{\alpha,f}-z_d) & \text{in} \ (0,T)\times \Omega,\\
\div \phi = 0 & \text{in} \ (0,T)\times \Omega,\\
\phi\cdot \nu=0 & \text{on} \ (0,T)\times \partial \Omega,\\
\left[2\mu D(\phi)\nu+\alpha \phi+(b\cdot \nu) \phi\right]_\tau = 0 & \text{on} \ (0,T)\times \partial \Omega,\\
\phi(T,\cdot)=0 & \text{in} \ \Omega.
\end{array}
\right.
\end{equation}
We denote by $(\phi_{f},\pi_{f})$ and by $(\phi_{\alpha,f},\pi_{\alpha,f})$ the corresponding solutions.
These adjoint systems are related to the optimal control problems \eqref{ns0.1} and \eqref{nh1.1} (see \cref{T01} below). 
Before giving these relations, let us state and prove the following important result:
\begin{Proposition}\label{P02}
Assume \eqref{1556},
and that $f,f_\alpha\in \mathcal{U}$ with
$$
f_\alpha \rightharpoonup
 f \quad \text{weakly in}\quad L^2(0,T;L^2(\omega)) \quad \text{as} \ \alpha \to \infty.
$$
Then, the solutions 
$(u_f,p_f)$, $(u_{\alpha,f_\alpha},p_{\alpha,f_\alpha})$, $(\phi_{f},\pi_{f})$ and $(\phi_{\alpha,f},\pi_{\alpha,f})$
of respectively \eqref{ns0.0}, \eqref{ns1.0},  \eqref{ns3.2}, \eqref{ns3.3}, 
satisfy, as $\alpha \to \infty$
\begin{equation}\label{ns1.6}
u_{\alpha,f_\alpha} \rightharpoonup
 u_f \quad \text{weakly * in}\quad L^2(0,T;H^1(\Omega))\cap L^\infty(0,T;L^2(\Omega)),
\end{equation}
\begin{equation}\label{ns1.7}
\partial_t u_{\alpha,f_\alpha} \rightharpoonup
\partial_t u_f \quad \text{weakly in}\quad L^{4/3}(0,T;V_0^{-1}),
\end{equation}
\begin{equation}\label{ns1.9}
u_{\alpha,f_\alpha} \to
 u_f \quad \text{strongly in}\quad L^2(0,T;L^2(\Omega)),
\end{equation}
\begin{equation}\label{ns1.8}
\sqrt{\alpha} (u_{\alpha,f_\alpha}-b) \to
 0 \quad \text{strongly in}\quad L^2(0,T;L^2(\partial \Omega)),
\end{equation}
\begin{equation}\label{ns3.0}
\phi_{\alpha,f_\alpha} \rightharpoonup
 \phi_f \quad \text{weakly * in}\quad L^2(0,T;H^1(\Omega))\cap L^\infty(0,T;L^2(\Omega)),
\end{equation}
\begin{equation}\label{ns3.1}
\partial_t \phi_{\alpha,f_\alpha} \rightharpoonup
\partial_t \phi_f \quad \text{weakly in}\quad L^{4/3}(0,T;V_0^{-1}),
\end{equation}
\begin{equation}\label{ns3.4}
\phi_{\alpha,f_\alpha} \to
 \phi_f \quad \text{strongly in}\quad L^2(0,T;L^2(\Omega)),
\end{equation}
\begin{equation}\label{ns1.8bis}
\sqrt{\alpha} \phi_{\alpha,f_\alpha} \to
 0 \quad \text{strongly in}\quad L^2(0,T;L^2(\partial \Omega)).
\end{equation}
\end{Proposition}
\begin{proof}
From \cref{P01}, we already know that the sequence $(u_{\alpha,f_\alpha})_\alpha$ is bounded in $L^2(0,T;H^1(\Omega))\cap L^\infty(0,T;L^2(\Omega))$
and that \eqref{ns1.8} holds.
Let us consider $\varphi\in L^4(0,T;V^1_0)$. Then
\begin{equation}\label{ns4.2}
\langle \partial_t u_\alpha,\varphi\rangle
= -\int_0^T \int_\Omega [(u_\alpha\cdot \nabla) u_\alpha]\cdot \varphi \ dx \ dt- \int_0^T \int_\Omega 2\mu D(u_\alpha): D(\varphi) \ dx \ dt+ \int_0^T \int_\omega f_\alpha \cdot \varphi \ dx \ dt.
\end{equation}
We have
\begin{equation}\label{ns4.3}
\left| \int_0^T \int_\Omega 2\mu D(u_\alpha): D(\varphi) \ dx \ dt \right| \leq C\|u_\alpha\|_{L^2(0,T;H^1(\Omega))} \|\varphi\|_{L^4(0,T;V^1_0)}
\end{equation}
and
\begin{multline}\label{ns4.4}
\left| \int_0^T \int_\Omega [(u_\alpha\cdot \nabla) u_\alpha]\cdot \varphi \ dx \ dt \right| \leq \int_0^T \|u_\alpha\|_{L^3(\Omega)} \|\nabla u_\alpha\|_{L^2(\Omega)} \|\varphi\|_{L^6(\Omega)}\ dt
\\
\leq C\int_0^T \|u_\alpha\|_{L^2(\Omega)}^{1/2}  \|u_\alpha\|_{H^1(\Omega)}^{3/2} \|\varphi\|_{H^1(\Omega)}\ dt
\\
\leq C\|u_\alpha\|_{L^\infty(0,T;L^2(\Omega))}^{1/2}  \|u_\alpha\|_{L^2(0,T;H^1(\Omega))}^{3/2}  \|\varphi\|_{L^4(0,T;V^1_0)}.
\end{multline}
Gathering \eqref{ns4.2}, \eqref{ns4.3} and \eqref{ns4.4}, we deduce that $(\partial_t u_{\alpha,f_\alpha})_\alpha$ is bounded in $L^{4/3}(0,T;V^{-1}_0)$.
Using the Banach-Alaoglu theorem combined with the Aubin-Lions compactness result (see, for instance, \cite[p. 271]{Temam}), we deduce that, up to a subsequence,
\begin{equation}\label{ns1.6bis}
u_{\alpha,f_\alpha} \rightharpoonup
 U \quad \text{weakly * in}\quad L^2(0,T;H^1(\Omega))\cap L^\infty(0,T;L^2(\Omega)),
\end{equation}
\begin{equation}\label{ns1.7bis}
\partial_t u_{\alpha,f_\alpha} \rightharpoonup
\partial_t U \quad \text{weakly in}\quad L^{4/3}(0,T;V_0^{-1}),
\end{equation}
\begin{equation}\label{ns1.9bis}
u_{\alpha,f_\alpha} \to
 U \quad \text{strongly in}\quad L^2(0,T;L^2(\Omega)).
\end{equation}

Now, let us consider $\varphi\in C^{\infty}_c([0,T)\times \Omega)$, $\div \varphi=0$. Multiplying the first equation of \eqref{ns1.0} by $\varphi$ and integrating by parts, we deduce that
\begin{multline*}
-\int_0^T\int_{\Omega} \partial_t \varphi \cdot u_\alpha \ dx\ dt-\int_0^T\int_{\Omega} [(u_\alpha\cdot \nabla)\varphi] \cdot u_\alpha \ dx\ dt+\int_0^T\int_{\Omega} 2\mu D(\varphi): D(u_\alpha) \ dx\ dt
\\
=\int_0^T\int_{\omega} f_\alpha \cdot \varphi \ dx\ dt+\int_\Omega \varphi(0,\cdot)\cdot a\ dx.
\end{multline*}

Using \eqref{ns1.6bis}, \eqref{ns1.7bis}, \eqref{ns1.9bis}, we deduce that $U\in L^2(0,T;H^1(\Omega))\cap L^\infty(0,T;L^2(\Omega))$ satisfies
\begin{multline*}
-\int_0^T\int_{\Omega} \partial_t \varphi \cdot U \ dx\ dt-\int_0^T\int_{\Omega} [(U\cdot \nabla)\varphi] \cdot U \ dx\ dt+\int_0^T\int_{\Omega} 2\mu D(\varphi): D(U) \ dx\ dt
\\
=\int_0^T\int_{\omega} f \cdot \varphi \ dx\ dt+\int_\Omega \varphi(0,\cdot)\cdot a\ dx,
\end{multline*}
with
$$
\div U=0, \quad U=b \quad \text{on} \quad (0,T)\times \partial \Omega.
$$
It means that $U$ is a weak solution of \eqref{ns0.0}. Using the weak-strong uniqueness (see \cite[pp. 298-299]{Temam}), we deduce that $U=u_f$.

The proof for the adjoint systems is similar: first we multiply the first equation of \eqref{ns3.3} by $\phi_{\alpha,f_\alpha}$:
\begin{multline}\label{ns6.1}
- \frac{1}{2} \frac{d}{dt} \int_{\Omega} |\phi_{\alpha,f_\alpha}|^2 \ dx 
+\int_{\Omega} (\nabla u_{\alpha,f_\alpha})^\top \phi_{\alpha,f_\alpha} \cdot \phi_{\alpha,f_\alpha}\ dx 
-\int_{\partial \Omega} (u_{\alpha,f_\alpha}\cdot \nu) \frac{|\phi_{\alpha,f_\alpha}|^2}{2}\ d\gamma
\\
+2\mu \int_{\Omega} |D(\phi_{\alpha,f_\alpha})|^2 \ dx 
+ \int_{\partial \Omega}  (\alpha +b\cdot \nu )\left|[\phi_{\alpha,f_\alpha}]_{\tau} \right|^2\ d\gamma
= 
\int_{\Omega} (u_{\alpha,f}-z_d)\cdot \phi_{\alpha,f_\alpha}\ dx.
\end{multline}
Then, integrating the above relation in $(t,T)$, we find
\begin{multline}\label{ns6.2}
 \frac{1}{2} \int_{\Omega} |\phi_{\alpha,f_\alpha}(t,\cdot)|^2 \ dx 
+2\mu \int_t^T \int_{\Omega} |D(\phi_{\alpha,f_\alpha})|^2 \ dx \ ds
+\int_t^T \int_{\Omega} (\nabla u_{\alpha,f_\alpha})^\top \phi_{\alpha,f_\alpha} \cdot \phi_{\alpha,f_\alpha}\ dx \ ds
\\
+\int_t^T \int_{\partial \Omega} (\alpha + \frac{1}{2} b\cdot \nu) |\phi_{\alpha,f_\alpha}|^2 \ d\gamma
= 
\int_t^T \int_{\Omega} (u_{\alpha,f_\alpha}-z_d)\cdot \phi_{\alpha,f_\alpha}\ dx\ ds.
\end{multline}
Using \eqref{1556}, H\"older's inequality, the Sobolev embedding $H^1(\Omega)\subset L^6(\Omega)$ and the Korn inequality, we deduce
\begin{multline}\label{ns6.3}
 \frac{1}{2} \int_{\Omega} |\phi_{\alpha,f_\alpha}(t,\cdot)|^2 \ dx 
+\mu \int_t^T \int_{\Omega} |D(\phi_{\alpha,f_\alpha})|^2 \ dx \ ds
+\frac{\alpha}{2} \int_t^T \int_{\partial \Omega} |\phi_{\alpha,f_\alpha}|^2 \ d\gamma
\\
\leq 
\frac{1}{2} \int_t^T \int_{\Omega} |u_{\alpha,f_\alpha}-z_d|^2\ dx\ ds
+C\int_t^T  (1+\|\nabla u_{\alpha,f_\alpha}\|_{L^2(\Omega)}^4)  \|\phi_{\alpha,f_\alpha}\|_{L^2(\Omega)}^2\ ds.
\end{multline}
Since $\left(u_{\alpha,f_\alpha}\right)_\alpha$ is bounded in $L^\infty(0,T;H^1(\Omega))$, we deduce from Gr\"onwall's lemma that
$\left(\phi_{\alpha,f_\alpha}\right)_\alpha$ is bounded in $L^\infty(0,T;L^2(\Omega))\cap L^2(0,T;H^1(\Omega))$ and that
$\left(\sqrt{\alpha} \phi_{\alpha,f_\alpha}\right)_\alpha$ is bounded in $L^2(0,T;L^2(\partial\Omega))$. 
Then we multiply the first equation of \eqref{ns3.3} by $\varphi\in L^4(0,T;V^1_0)$:
\begin{multline}\label{ns6.4}
\langle \partial_t \phi_{\alpha,f_\alpha},\varphi\rangle =
\int_0^T\int_{\Omega} (\nabla u_{\alpha,f_\alpha})^\top \phi_{\alpha,f_\alpha} \cdot \varphi\ dx \ dt
-\int_0^T\int_{\Omega}(u_{\alpha,f_\alpha}\cdot \nabla) \phi_{\alpha,f_\alpha}\cdot \varphi\ dx \ dt
\\
+\int_0^T\int_{\Omega} 2\mu D(\phi_{\alpha,f_\alpha}): D(\varphi) \ dx \ dt
-\int_0^T\int_{\Omega}(u_{\alpha,f_\alpha}-z_d) \cdot \varphi\ dx \ dt.
\end{multline}
We have
\begin{multline}\label{ns6.5}
\left|\int_0^T\int_{\Omega} (\nabla u_{\alpha,f_\alpha})^\top \phi_{\alpha,f_\alpha} \cdot \varphi\ dx \ dt\right|
\leq C \int_0^T \|u_{\alpha,f_\alpha}\|_{H^1(\Omega)} \| \phi_{\alpha,f_\alpha} \|_{L^2(\Omega)}^{1/2} \| \phi_{\alpha,f_\alpha} \|_{H^1(\Omega)}^{1/2} 
\|\varphi\|_{H^1(\Omega)} \ dt
\\
\leq C\|u_{\alpha,f_\alpha}\|_{L^2(0,T;H^1(\Omega))} \| \phi_{\alpha,f_\alpha} \|_{L^\infty(0,T;L^2(\Omega))}^{1/2}
\| \phi_{\alpha,f_\alpha} \|_{L^2(0,T;H^1(\Omega))}^{1/2} 
\|\varphi\|_{L^4(0,T;H^1(\Omega))}
\\
\leq C\|\varphi\|_{L^4(0,T;H^1(\Omega))},
\end{multline}
and
\begin{multline}\label{ns6.6}
\left| \int_0^T\int_{\Omega}(u_{\alpha,f_\alpha}\cdot \nabla) \phi_{\alpha,f_\alpha}\cdot \varphi\ dx \ dt \right|
\leq C \int_0^T \|u_{\alpha,f_\alpha}\|_{L^2(\Omega)}^{1/2}\|u_{\alpha,f_\alpha}\|_{H^1(\Omega)}^{1/2} \| \phi_{\alpha,f_\alpha} \|_{H^1(\Omega)}
\|\varphi\|_{H^1(\Omega)} \ dt
\\
\leq C\|u_{\alpha,f_\alpha}\|_{L^\infty(0,T;L^2(\Omega))}^{1/2} \|u_{\alpha,f_\alpha}\|_{L^2(0,T;H^1(\Omega))}^{1/2} 
\| \phi_{\alpha,f_\alpha} \|_{L^2(0,T;H^1(\Omega))}
\|\varphi\|_{L^4(0,T;H^1(\Omega))}
\\
\leq C\|\varphi\|_{L^4(0,T;H^1(\Omega))}.
\end{multline}
Combining  \eqref{ns6.4}, \eqref{ns6.5},  \eqref{ns6.6} with standard estimates, we deduce that the sequence
$\left(\partial_t \phi_{\alpha,f_\alpha}\right)_\alpha$ is bounded in $L^{4/3}(0,T;V_0^{-1})$.
Using the Banach-Alaoglu theorem and the Aubin-Lions compactness result (see, for instance, \cite[p. 271]{Temam}), we deduce that, up to a subsequence,
\begin{equation}\label{ns7.0}
\phi_{\alpha,f_\alpha} \rightharpoonup
 \Phi \quad \text{weakly * in}\quad L^2(0,T;H^1(\Omega))\cap L^\infty(0,T;L^2(\Omega)),
\end{equation}
\begin{equation}\label{ns7.1}
\partial_t \phi_{\alpha,f_\alpha} \rightharpoonup
\partial_t \Phi \quad \text{weakly in}\quad L^{4/3}(0,T;V_0^{-1}),
\end{equation}
\begin{equation}\label{ns7.2}
\phi_{\alpha,f_\alpha} \to
 \Phi \quad \text{strongly in}\quad L^2(0,T;L^2(\Omega)).
\end{equation}

Now, let us consider $\varphi\in C^{\infty}_c((0,T]\times \Omega)$, $\div \varphi=0$. Multiplying the first equation of \eqref{ns3.3} by $\varphi$ and integrating by parts, we deduce that
\begin{multline*}
 \int_0^T\int_{\Omega} \phi_{\alpha,f_\alpha} \cdot\partial_t \varphi \ dx \ dt
+\int_0^T\int_{\Omega} (\nabla u_{\alpha,f_\alpha})^\top \phi_{\alpha,f_\alpha} \cdot \varphi\ dx \ dt
-\int_0^T\int_{\Omega}(u_{\alpha,f_\alpha}\cdot \nabla) \phi_{\alpha,f_\alpha}\cdot \varphi\ dx \ dt
\\
+\int_0^T\int_{\Omega} 2\mu D(\phi_{\alpha,f_\alpha}): D(\varphi) \ dx \ dt
=\int_0^T\int_{\Omega}(u_{\alpha,f_\alpha}-z_d) \cdot \varphi\ dx \ dt.
\end{multline*}

Using \eqref{ns7.0}, \eqref{ns7.2}, \eqref{ns1.6bis}, \eqref{ns1.9bis}, we deduce that $\Phi\in L^2(0,T;H^1(\Omega))\cap L^\infty(0,T;L^2(\Omega))$ satisfies
\begin{multline*}
 \int_0^T\int_{\Omega} \Phi \cdot\partial_t \varphi \ dx \ dt
+\int_0^T\int_{\Omega} (\nabla u_{f})^\top \Phi \cdot \varphi\ dx \ dt
-\int_0^T\int_{\Omega}(u_{f}\cdot \nabla) \Phi\cdot \varphi\ dx \ dt
\\
+\int_0^T\int_{\Omega} 2\mu D(\Phi): D(\varphi) \ dx \ dt
=\int_0^T\int_{\Omega}(u_{f}-z_d) \cdot \varphi\ dx \ dt,
\end{multline*}
with
$$
\div \Phi=0, \quad \Phi=0 \quad \text{on} \quad (0,T)\times \partial \Omega.
$$
It means that $\Phi$ is a weak solution of \eqref{ns3.2}. 
Using the weak-strong uniqueness of the linear system  \eqref{ns3.2}, we deduce that $\Phi=\phi_f$.
\end{proof}

\section{Existence for the optimal control problems}\label{sec_exi}
This section is devoted to the following classical result, showing the existence of an optimal control and giving a first order necessary condition in terms of the adjoint states:
\begin{Theorem}\label{T01}
The problems \eqref{ns0.1} and \eqref{nh1.1} admit at least a solution. Moreover, if $\overline{f}\in \mathcal{U}$ is a solution to \eqref{ns0.1}, then
\begin{equation}\label{ns3.5}
\int_0^T\int_{\omega} (g-\overline{f})\cdot \phi_{\overline{f}} \ dx \ dt+M\int_0^T\int_{\omega} (g-\overline{f})\cdot \overline{f} \ dx \ dt\geq 0 \quad \forall g\in \mathcal{U}.
\end{equation}
Similarly, if $\overline{f_\alpha}\in \mathcal{U}$ is a solution to \eqref{nh1.1}, then
\begin{equation}\label{ns3.6}
\int_0^T\int_{\omega} (g-\overline{f_\alpha})\cdot \phi_{\alpha,\overline{f_\alpha}} \ dx \ dt+M\int_0^T\int_{\omega} (g-\overline{f_\alpha})\cdot \overline{f_\alpha} \ dx \ dt\geq 0 \quad \forall g\in \mathcal{U}.
\end{equation}
\end{Theorem}
\begin{proof}
The proof is quite standard (see for instance \cite{Casas}) and we only sketch the proof in the case of the Navier boundary conditions for the sake of completeness. 
First to show the existence of a solution of \eqref{nh1.1} , we consider a minimizing sequence $(f_{\alpha,k})\in \mathcal{U}$,
$$
J_\alpha(f_{\alpha,k})\to \inf_{\mathcal{U}} J_\alpha.,
$$
as $k\to \infty$.
Using the definition \eqref{nh1.2} of $J_\alpha$, we deduce from the above limit that $(f_{\alpha,k})_k$ is a bounded sequence of $L^2(0,T;L^2(\omega))$. Thus
$$
f_{\alpha,k} \rightharpoonup
 f_{\alpha} \quad \text{weakly in}\quad L^2(0,T;L^2(\omega)).
$$
Then following the same steps of the proof of \cref{P02}, we can show that 
\begin{equation}\label{ns4.6}
u_{\alpha,f_{\alpha,k}} \rightharpoonup
 U_{\alpha} \quad \text{weakly * in}\quad L^2(0,T;H^1(\Omega))\cap L^\infty(0,T;L^2(\Omega)),
\end{equation}
\begin{equation}\label{ns4.7}
\partial_t u_{\alpha,f_{\alpha,k}} \rightharpoonup
\partial_t U_{\alpha} \quad \text{weakly in}\quad L^{4/3}(0,T;V_0^{-1}),
\end{equation}
\begin{equation}\label{ns4.9}
u_{\alpha,f_{\alpha,k}} \to
 U_{\alpha} \quad \text{strongly in}\quad L^2(0,T;L^2(\Omega)),
\end{equation}
and passing to the limit in the weak formulation of \eqref{ns1.0}, we deduce that $U_{\alpha}$ is a weak solution of \eqref{ns1.0} associated with $f_{\alpha}$. 
Using the weak-strong uniqueness (see \cite[pp. 298-299]{Temam}), we deduce that $U_{\alpha}=u_{\alpha,f_{\alpha}}$.
Moreover,
$$
J_{\alpha}(f_{\alpha}) \leq \liminf_k J_{\alpha}(f_{\alpha,k})=\inf_{\mathcal{U}} J_\alpha.
$$
To obtain the first order optimality condition, we use the Gateaux-differentiability of $J_\alpha$ and of the state $\Lambda : f\mapsto (u_{\alpha,f},p_{\alpha,f})$.
More precisely, by denoting by
$$
d \Lambda_{f}(g):=(v,q)
$$
the derivative of $\Lambda$ in $f$ and in the direction $g$, we can check that 
\begin{equation}\label{ns5.0}
\left\{
\begin{array}{rl}
\partial_t v +(u_{\alpha,f}\cdot \nabla) v+(v\cdot \nabla) u_{\alpha,f}- \div \sigma(v,q)= g 1_{\omega} & \text{in} \ (0,T)\times \Omega,\\
\div v = 0 & \text{in} \ (0,T)\times \Omega,\\
v\cdot \nu=0 & \text{on} \ (0,T)\times \partial \Omega,\\
\left[2\mu D(v)\nu+\alpha v\right]_\tau = 0 & \text{on} \ (0,T)\times \partial \Omega,\\
v(0,\cdot)=0 & \text{in} \ \Omega.
\end{array}
\right.
\end{equation}
Then we have from \eqref{nh1.2}
\begin{equation}\label{ns5.1}
(d J_\alpha)_{f}(g)= \int_0^T\int_{\Omega} (u_{\alpha,f}-z_d)\cdot v \ dx\ dt+M \int_0^T\int_\omega f\cdot g \ dx\ dt.
\end{equation}
Multiplying the first equation of \eqref{ns3.3} by $v$:
\begin{multline*}
\int_0^T\int_{\Omega} (u_{\alpha,f}-z_d)\cdot v \ dx\ dt
=
\int_0^T\int_{\Omega} -\partial_t \phi_{\alpha,f} \cdot v \ dx\ dt
+
\int_0^T\int_{\Omega} (\nabla u_{\alpha,f})^\top \phi_{\alpha,f} \cdot v \ dx\ dt
\\
-
\int_0^T\int_{\Omega} (u_{\alpha,f}\cdot \nabla) \phi_{\alpha,f}\cdot v \ dx\ dt
- 
\int_0^T\int_{\Omega} \div \sigma(\phi_{\alpha,f},\pi_{\alpha,f})\cdot v \ dx\ dt
\\
=
\int_0^T\int_{\omega} \phi_{\alpha,f} \cdot g\ dx\ dt
+\int_0^T\int_{\partial \Omega} -(u_{\alpha,f}\cdot \nu) \phi_{\alpha,f}\cdot v
+\sigma(v,q)\nu\cdot \phi_{\alpha,f}
-\sigma(\phi_{\alpha,f},\pi_{\alpha,f})\nu \cdot v
\ d\gamma\ dt
\\
=
\int_0^T\int_{\omega} \phi_{\alpha,f} \cdot g\ dx\ dt.
\end{multline*}

Now, since $\mathcal{U}$ is convex set, if $\overline{f_\alpha}$ is a solution of \eqref{nh1.1},
we have
$$
(d J_\alpha)_{\overline{f_\alpha}}(g-\overline{f_\alpha}) \geq 0 \quad \forall g\in \mathcal{U}.
$$
With the above computation, this writes \eqref{ns3.6}.
\end{proof}

\begin{Remark}
Note that conditions \eqref{ns3.5} and \eqref{ns3.6} can be written as
$$
\overline{f}=\mathcal{P}_{\mathcal{U}}\left(-\frac{1}{M} \phi_{\overline{f}}\right)
\quad 
\text{and}
\quad
\overline{f_\alpha}=\mathcal{P}_{\mathcal{U}}\left(-\frac{1}{M} \phi_{\alpha,\overline{f_\alpha}}\right)
$$
where $\mathcal{P}_{\mathcal{U}}$ is the projection on the convex set $\mathcal{U}$.
\end{Remark}

\section{Proof of the main result}\label{sec_main}
We are now in a position to prove \cref{Thmain}:
\begin{proof}[Proof of \cref{Thmain}]
Let us consider a family of optimal controls $\overline{f_\alpha}$ of problem \eqref{nh1.1}. For simplicity, in this proof we use the notation $f_\alpha$ instead of $\overline{f_\alpha}$.
First we note that 
\begin{equation}\label{ns3.7}
\frac{M}{2} \int_0^T\int_\omega |f_\alpha|^2 \ dx\ dt \leq J_\alpha({f_\alpha}) \leq J_\alpha(\widehat{f})
=\frac 12 \int_0^T\int_{\Omega} |u_{\alpha,\widehat{f}}-z_d|^2 \ dx\ dt+
\frac{M}{2} \int_0^T\int_\omega |\widehat{f}|^2 \ dx\ dt.
\end{equation}
Applying \cref{P02}, we deduce that $\left(u_{\alpha,\widehat{f}}\right)_{\alpha}$ is bounded in $L^2(0,T;L^2(\Omega))$ and thus from \eqref{ns3.7} we obtain that
$(f_{\alpha})_{\alpha}$ is bounded in $L^2(0,T;L^2(\omega))$.
Consequently, there exists $f\in L^2(0,T;L^2(\omega))$ up to a subsequence,
\begin{equation}\label{ns3.8}
f_\alpha \rightharpoonup
 f \quad \text{weakly in}\quad L^2(0,T;L^2(\omega)).
\end{equation}

Using that $\mathcal{U}$ is convex and closed in $L^2(0,T;L^2(\omega))$ 
(see \eqref{ns1.5}), we deduce that it is also closed for the weak topology and thus $f\in \mathcal{U}$.

We can thus apply \cref{P02} and we obtain relations \eqref{ns1.6}--\eqref{ns1.8bis}.
In particular, from \eqref{ns1.9}, we deduce 
\begin{equation}\label{ns6.0}
\frac 12 \int_0^T\int_{\Omega} |u_{\alpha,f_\alpha}-z_d|^2 \ dx\ dt\to \frac 12 \int_0^T\int_{\Omega} |u_f-z_d|^2 \ dx\ dt
\end{equation}
and \eqref{ns3.8} implies
$$
\int_0^T\int_\omega |f|^2 \ dx\ dt \leq \liminf_{\alpha\to\infty} \int_0^T\int_\omega |f_\alpha|^2 \ dx\ dt.
$$
Combining the two last relations, we obtain
\begin{equation}\label{ns3.9}
J(f)\leq \liminf_{\alpha\to\infty} J_\alpha(f_\alpha).
\end{equation}
On the other hand, by definition of $f_\alpha$, we have
$$
J_\alpha(f_\alpha) \leq J_\alpha(f)
$$
and applying again \cref{P02}, we deduce that
$$
u_{\alpha,f} \to
 u_f \quad \text{strongly in}\quad L^2(0,T;L^2(\Omega)),
$$
so that
$$
\limsup_{\alpha\to\infty} J_\alpha(f_\alpha)\leq \limsup_{\alpha\to\infty} J_\alpha(f) = J(f).
$$
The above relation and \eqref{ns3.9} yields that
\begin{equation}\label{ns4.0}
\lim_{\alpha\to\infty} J_\alpha(f_\alpha) =J(f).
\end{equation}

Moreover since
$$
J_\alpha(f_\alpha)\leq J_\alpha(g) \quad \forall g\in \mathcal{U}
$$
and since, by using again \cref{P02}
$$
\lim_{\alpha\to\infty} J_\alpha(g) = J(g),
$$
we deduce that $f$ is a solution to \eqref{ns0.1}.

From \eqref{ns6.0} and \eqref{ns4.0}, we deduce that
$$
\int_0^T\int_\omega |f_\alpha|^2 \ dx\ dt \to \int_0^T\int_\omega |f|^2 \ dx\ dt.
$$
as $\alpha\to\infty$ and thus we obtain \eqref{ns5.2} from \eqref{ns3.8}.
\end{proof}

\section{The bidimensional case}\label{sec2d}
In that case, we can work with weak solutions (that are unique) for systems \eqref{ns0.0}, \eqref{ns1.0}. More precisely,
we assume (instead of \eqref{cc1} and \eqref{cc2})
\begin{equation}\label{cc3}
a\in L^2(\Omega), \quad \div a=0, \quad a\cdot \nu=b(0,\cdot)\cdot \nu \quad \text{on} \ \partial \Omega
\end{equation}
and
$$
b\in H^1(0,T;L^2(\Omega))\cap C^0([0,T);H^1(\Omega))\cap L^2(0,T;H^2(\Omega)), \quad \div b =0 \quad \text{in} \ (0,T)\times \Omega.
$$
In particular,
\begin{equation}\label{cc4}
b\in L^2(0,T;H^{3/2}(\partial \Omega))\cap H^{1/4}(0,T;L^2(\partial \Omega)), \quad
\int_{\partial \Omega} b(t,\cdot)\cdot \nu \ d\gamma=0 \quad (t\in [0,T]).
\end{equation}

The weak solutions of \eqref{ns0.0} satisfies 
$$
u_f\in L^2(0,T;H^1(\Omega))\cap C([0,T];L^2(\Omega)) \cap H^{1}(0,T;V^{-1}_0),
$$
$$
\left\{
\begin{array}{rl}
\div u_f = 0 & \text{in} \ (0,T)\times \Omega,\\
u_f=b & \text{on} \ (0,T)\times \partial \Omega,\\
\end{array}
\right.
$$
and 
\begin{multline}\label{1900}
-\int_0^T\int_{\Omega} \partial_t \varphi \cdot u_f \ dx\ dt-\int_0^T\int_{\Omega} [(u_f\cdot \nabla)\varphi] \cdot u_f \ dx\ dt+\int_0^T\int_{\Omega} 2\mu D(\varphi): D(u_f) \ dx\ dt
\\
=\int_0^T\int_{\omega} f \cdot \varphi \ dx\ dt+\int_\Omega \varphi(0,\cdot)\cdot a\ dx,
\end{multline}
for any $\varphi\in C^1_c([0,T);V^1_0)$. We recall that $V^1_0$ is defined by \eqref{1448}.

The definition of weak solutions for \eqref{ns1.0} is similar:
$$
u_{\alpha,f}\in L^2(0,T;H^1(\Omega))\cap C([0,T];L^2(\Omega)) \cap H^{1}(0,T;V^{-1}_\nu),
$$
$$
\left\{
\begin{array}{rl}
\div u_{\alpha,f} = 0 & \text{in} \ (0,T)\times \Omega,\\
u_{\alpha,f}\cdot \nu=b\cdot \nu& \text{on} \ (0,T)\times \partial \Omega,\\
\end{array}
\right.
$$
and 
\begin{multline}\label{1858}
-\int_0^T\int_{\Omega} \partial_t \varphi \cdot u_{\alpha,f} \ dx\ dt-\int_0^T\int_{\Omega} [(u_{\alpha,f}\cdot \nabla)\varphi] \cdot u_{\alpha,f} \ dx\ dt
+\int_0^T\int_{\Omega} 2\mu D(\varphi): D(u_{\alpha,f}) \ dx\ dt
\\
+\int_0^T\int_{\partial \Omega}  \left[(b\cdot \nu) u_\tau +\alpha (u-b)_\tau\right] \cdot \varphi_\tau \ d\gamma \ dt
=\int_0^T\int_{\omega} f_\alpha \cdot \varphi \ dx\ dt+\int_\Omega \varphi(0,\cdot)\cdot a\ dx,
\end{multline}
for any $\varphi\in C^1_c([0,T);V^1_\nu)$ where
\begin{equation}\label{1448-2d}
V_\nu^1:=\left\{u\in H^1(\Omega) \ ; \ \div u=0, \quad u\cdot \nu=0 \ \text{on} \ \partial \Omega\right\},
\end{equation}
$$
V_\nu^{-1}:=(V_\nu^1)'.
$$

With this framework, the hypotheses on $\mathcal{U}$ are weaker than in the 3d case: we only assume that
\begin{equation}\label{ns1.5-2d}
\mathcal{U} \ \text{is a closed convex non empty subset of} \ L^2(0,T;L^2(\omega))
\end{equation}
instead of \eqref{ns1.5}.

With these assumptions, \cref{T01} holds true with the same proof. The main result becomes
\begin{Theorem}\label{Thmain-2d}
Assume that $(a,b)$ and $\mathcal{U}$ satisfy the above hypotheses.

Then for any $\alpha$ large enough, the problem \eqref{nh1.1} admits a solution $\overline{f_\alpha}$ and
there exist $\overline{f}\in \mathcal{U}$ and 
a sequence such that as $\alpha \to \infty$
\begin{equation}\label{ns5.2-2d}
\overline{f_\alpha} \to \overline{f} \quad \text{strongly in} \quad L^2(0,T;L^2(\omega)),
\end{equation}
and $\overline{f}$ is a solution of \eqref{ns0.1}.
Moreover, the corresponding solutions $(u_{\overline{f}},p_{\overline{f}})$ and $(u_{\alpha,\overline{f_\alpha}},p_{\alpha,\overline{f_\alpha}})$
of \eqref{ns0.0} and \eqref{ns1.0}
satisfy
\begin{equation}\label{ns5.3-2d}
u_{\alpha,\overline{f_\alpha}} \rightharpoonup
 u_{\overline{f}} \quad \text{weakly * in}\quad L^2(0,T;H^1(\Omega))\cap L^\infty(0,T;L^2(\Omega)),
\end{equation}
\begin{equation}\label{ns5.4-2d}
\partial_t u_{\alpha,\overline{f_\alpha}} \rightharpoonup
\partial_t u_{\overline{f}} \quad \text{weakly in}\quad L^{2}(0,T;V_0^{-1}),
\end{equation}
\begin{equation}\label{ns5.5-2d}
u_{\alpha,\overline{f_\alpha}} \to
 u_{\overline{f}} \quad \text{strongly in}\quad L^2(0,T;L^2(\Omega)),
\end{equation}
\begin{equation}\label{ns5.6-2d}
\sqrt{\alpha} (u_{\alpha,\overline{f_\alpha}}-b) \to
 0 \quad \text{strongly in}\quad L^2(0,T;L^2(\partial \Omega)).
\end{equation}
The solutions of the adjoint systems (defined by \eqref{ns3.2} and \eqref{ns3.3}) $(\phi_{f},\pi_{f})$ and $(\phi_{\alpha,f},\pi_{\alpha,f})$
satisfy
\begin{equation}\label{ns5.7-2d}
\phi_{\alpha,\overline{f_\alpha}} \rightharpoonup
 \phi_{\overline{f}} \quad \text{weakly * in}\quad L^2(0,T;H^1(\Omega))\cap L^\infty(0,T;L^2(\Omega)),
\end{equation}
\begin{equation}\label{ns5.8-2d}
\partial_t \phi_{\alpha,\overline{f_\alpha}} \rightharpoonup
\partial_t \phi_{\overline{f}} \quad \text{weakly in}\quad L^{4/3}(0,T;V_0^{-1}),
\end{equation}
\begin{equation}\label{ns5.9-2d}
\phi_{\alpha,\overline{f_\alpha}} \to
 \phi_{\overline{f}} \quad \text{strongly in}\quad L^2(0,T;L^2(\Omega)),
\end{equation}
\begin{equation}\label{ns5.6bis-2d}
\sqrt{\alpha} \phi_{\alpha,\overline{f_\alpha}} \to
 0 \quad \text{strongly in}\quad L^2(0,T;L^2(\partial \Omega)).
\end{equation}
\end{Theorem}
The proof of \cref{Thmain-2d} is the same as the proof of \cref{Thmain}, we only use the following result instead of \cref{P02}:
\begin{Proposition}\label{P10}
Assume \eqref{1556}
and that $f,f_\alpha\in \mathcal{U}$ with
\begin{equation}\label{1736}
f_\alpha \rightharpoonup
 f \quad \text{weakly in}\quad L^2(0,T;L^2(\omega)).
\end{equation}
Then, the (weak) solutions 
$(u_f,p_f)$, $(u_{\alpha,f_\alpha},p_{\alpha,f_\alpha})$, $(\phi_{f},\pi_{f})$ and $(\phi_{\alpha,f},\pi_{\alpha,f})$
of respectively \eqref{ns0.0}, \eqref{ns1.0},  \eqref{ns3.2}, \eqref{ns3.3}, 
satisfy \eqref{ns1.6}--\eqref{ns1.8bis} and 
\begin{equation}\label{ns1.7-2d}
\partial_t u_{\alpha,f_\alpha} \rightharpoonup
\partial_t u_f \quad \text{weakly in}\quad L^{2}(0,T;V_0^{-1}).
\end{equation}
\end{Proposition}
\begin{proof}
From \eqref{1736}, we deduce that $\left(f_\alpha\right)_\alpha$ is bounded in $L^2(0,T;L^2(\omega))$.
We set  
$$
w:=u_{\alpha,f_\alpha}-b, \quad g:=f_\alpha 1_\omega-\partial_t b -(b\cdot\nabla)b+\Delta b\in L^2(0,T;L^2(\Omega))
$$
that satisfy
\begin{equation}\label{ns2.0-2d}
\left\{
\begin{array}{rl}
\partial_t w +(w\cdot \nabla) w
+(b\cdot \nabla) w
+(w\cdot \nabla) b
- \div \sigma(w,p)
= g & \text{in} \ (0,T)\times \Omega,\\
\div w = 0 & \text{in} \ (0,T)\times \Omega,\\
w\cdot \nu=0 & \text{on} \ (0,T)\times \partial \Omega,\\
\left[2\mu D(w)\nu+\alpha w\right]_\tau = 0  & \text{on} \ (0,T)\times \partial \Omega,\\
w(0,\cdot)=a-b(0,\cdot) & \text{in} \ \Omega.
\end{array}
\right.
\end{equation}
To obtain our estimates, we multiply the first equation of \eqref{ns2.0-2d} by $w$ and we deduce
\begin{multline}\label{ns4.1-2d}
\frac{1}{2}\frac{d}{dt} \int_{\Omega} |w|^2 \ dx + \int_{\partial \Omega} \frac{b\cdot \nu}{2} w_\tau^2 \ d\gamma
+\int_{\Omega} [(w\cdot\nabla)b]\cdot w \ dx + 2\mu \int_{\Omega} |D(w)|^2 \ dx +\alpha \int_{\partial \Omega} w_\tau^2 \ d\gamma
\\
=\int_{\Omega} g\cdot w \ dx.
\end{multline}

Using H\"older's inequality, the Sobolev embedding $H^{1/2}(\Omega)\subset L^4(\Omega)$ and the Korn inequality, we deduce
\begin{multline}\label{1558-2d}
\left| \int_{\Omega} [(w\cdot\nabla)b]\cdot w \ dx\right|
\leq \|w\|_{L^4(\Omega)}^2 \|\nabla b\|_{L^2(\Omega)} 
\leq C \|w\|_{L^2(\Omega)} \|w\|_{H^1(\Omega)} \|b\|_{H^1(\Omega)}
\\
\leq \mu \|D(w)\|_{L^2(\Omega)}^2+
C(1+\|b\|_{H^1(\Omega)}^2) \|w\|_{L^2(\Omega)}^2.
\end{multline}

Using \eqref{1556},
we deduce from \eqref{ns4.1-2d} and from \eqref{1558-2d} that
\begin{multline}\label{ns1.2-2d}
\int_{\Omega} |w(t,\cdot)|^2 \ dx 
+ 2\mu \int_0^t \int_{\Omega} |D(w)|^2 \ dx \ ds
+\alpha \int_0^t \int_{\partial \Omega} w_\tau^2 \ d\gamma \ ds
\\
\leq 
\|a-b(0,\cdot)\|_{L^2(\Omega)}^2
+C\|g\|_{L^2(0,T;L^2(\Omega))}^2 
+C\int_0^t (1+\|b\|_{H^1(\Omega)}^2) \|w\|_{L^2(\Omega)}^2 \ ds,
\end{multline}
where the constants $C$ are independent of $\alpha$.

Using the Gr\"onwall lemma and the Korn lemma, we deduce 
\begin{multline}
\|w\|_{L^\infty(0,T;L^2(\Omega))}^2
+\|w\|_{L^2(0,T;H^1(\Omega))}^2
+\alpha \|w_\tau\|_{L^2(0,T;L^2(\partial \Omega))}^2
\\
\leq 
\left(
\|a-b(0,\cdot)\|_{L^2(\Omega)}^2
+C\|g\|_{L^2(0,T;L^2(\Omega))}^2 
\right)
\exp\left(C(T+\|b\|_{L^2(0,T;H^1(\Omega))}^2)\right),
\end{multline}
where the constants $C$ are independent of $\alpha$.

We deduce that the sequence $(u_{\alpha,f_\alpha})_\alpha$ is bounded in $L^2(0,T;H^1(\Omega))\cap L^\infty(0,T;L^2(\Omega))$
and that \eqref{ns1.8} holds.
Let us consider $\varphi\in L^2(0,T;V^1_0)$. Then
\begin{equation}\label{ns4.2-2d}
\langle \partial_t u_\alpha,\varphi\rangle
= \int_0^T \int_\Omega [(u_\alpha\cdot \nabla) \varphi]\cdot u_\alpha \ dx \ dt- \int_0^T \int_\Omega 2\mu D(u_\alpha): D(\varphi) \ dx \ dt+ \int_0^T \int_\omega f_\alpha \cdot \varphi \ dx \ dt.
\end{equation}
We have
\begin{equation}\label{ns4.3-2d}
\left| \int_0^T \int_\Omega 2\mu D(u_\alpha): D(\varphi) \ dx \ dt \right| \leq C\|u_\alpha\|_{L^2(0,T;H^1(\Omega))} \|\varphi\|_{L^4(0,T;V^1_0)}
\end{equation}
and
\begin{multline}\label{ns4.4-2d}
\left| \int_0^T \int_\Omega [(u_\alpha\cdot \nabla) \varphi]\cdot u_\alpha \ dx \ dt \right| 
\leq \int_0^T \|u_\alpha\|_{L^4(\Omega)}^2 \|\nabla \varphi\|_{L^2(\Omega)}\ dt
\leq C\int_0^T \|u_\alpha\|_{L^2(\Omega)} \|u_\alpha\|_{H^1(\Omega)} \|\varphi\|_{H^1(\Omega)}\ dt
\\
\leq C\|u_\alpha\|_{L^\infty(0,T;L^2(\Omega))}  \|u_\alpha\|_{L^2(0,T;H^1(\Omega))}  \|\varphi\|_{L^2(0,T;V^1_0)}.
\end{multline}
Gathering \eqref{ns4.2-2d}, \eqref{ns4.3-2d} and \eqref{ns4.4-2d}, we deduce that $(\partial_t u_{\alpha,f_\alpha})_\alpha$ is bounded in $L^{2}(0,T;V^{-1}_0)$.
Using the Banach-Alaoglu theorem combined with the Aubin-Lions compactness result (see, for instance, \cite[p. 271]{Temam}), we deduce that, up to a subsequence,
\begin{equation}\label{ns1.6bis-2d}
u_{\alpha,f_\alpha} \rightharpoonup
 U \quad \text{weakly * in}\quad L^2(0,T;H^1(\Omega))\cap L^\infty(0,T;L^2(\Omega)),
\end{equation}
\begin{equation}\label{ns1.7bis-2d}
\partial_t u_{\alpha,f_\alpha} \rightharpoonup
\partial_t U \quad \text{weakly in}\quad L^{2}(0,T;V_0^{-1}),
\end{equation}
\begin{equation}\label{ns1.9bis-2d}
u_{\alpha,f_\alpha} \to
 U \quad \text{strongly in}\quad L^2(0,T;L^2(\Omega)).
\end{equation}

Now, let us take $\varphi\in C^{\infty}_c([0,T)\times \Omega)$, $\div \varphi=0$
in \eqref{1858}.
Using \eqref{ns1.6bis-2d}, \eqref{ns1.7bis-2d}, \eqref{ns1.9bis-2d}, we deduce that $U\in L^2(0,T;H^1(\Omega))\cap L^\infty(0,T;L^2(\Omega))$ satisfies
\eqref{1900} with
$$
\div U=0, \quad U=b \quad \text{on} \quad (0,T)\times \partial \Omega.
$$
It means that $U$ is a weak solution of \eqref{ns0.0}. Using the uniqueness of weak solutions (see \cite[p. 294]{Temam}), we deduce that $U=u_f$.

The proof for the adjoint systems is the same as in the 3d case. 
\end{proof}

\begin{Remark}
From the Sobolev embeddings in the 2d case, we can improve the convergence \eqref{ns3.1} and obtain 
\begin{equation}\label{ns3.1-2dbis}
\partial_t \phi_{\alpha,f_\alpha} \rightharpoonup
\partial_t \phi_f \quad \text{weakly in}\quad L^{2-\varepsilon}(0,T;V_0^{-1}),
\end{equation}
for any $\varepsilon>0$.
The ``worst'' term to estimate is
\begin{multline}\label{ns6.5-2d}
\left|\int_0^T\int_{\Omega} (\nabla u_{\alpha,f_\alpha})^\top \phi_{\alpha,f_\alpha} \cdot \varphi\ dx \ dt\right|
\leq C \int_0^T \|u_{\alpha,f_\alpha}\|_{H^1(\Omega)} \| \phi_{\alpha,f_\alpha} \|_{L^2(\Omega)}^{(2-2\varepsilon)/(2-\varepsilon)} \| \phi_{\alpha,f_\alpha} \|_{H^1(\Omega)}^{\varepsilon/(2-\varepsilon)} 
\|\varphi\|_{H^1(\Omega)} \ dt
\\
\leq C\|u_{\alpha,f_\alpha}\|_{L^2(0,T;H^1(\Omega))} \| \phi_{\alpha,f_\alpha} \|_{L^\infty(0,T;L^2(\Omega))}^{(2-2\varepsilon)/(2-\varepsilon)}
\| \phi_{\alpha,f_\alpha} \|_{L^2(0,T;H^1(\Omega))}^{\varepsilon/(2-\varepsilon)} 
\|\varphi\|_{L^{(2-\varepsilon)/(1-\varepsilon)}(0,T;H^1(\Omega))}
\\
\leq C\|\varphi\|_{L^{(2-\varepsilon)/(1-\varepsilon)}(0,T;H^1(\Omega))}.
\end{multline}
\end{Remark}

\bibliographystyle{plain}
\bibliography{reference}

\end{document}